\documentclass{article}
\usepackage[dvips]{graphicx}
\usepackage[T1]{fontenc}
\usepackage[latin1]{inputenc}
\usepackage[english]{babel}
\usepackage{psfrag}
\usepackage{times}
\usepackage{amsmath}
\usepackage{amssymb}
\usepackage{amscd}
\usepackage{latexsym}
\usepackage{tabularx}
\usepackage{here}
\usepackage{color}
\usepackage{epsfig}
\usepackage{verbatim}

\author{
  Yuehua Bu\thanks{
{\it Departement of Mathematics, Zhejiang Normal University, Jinhua 321004, P. R. China}, \texttt{yhbu@zjnu.cn}} \and
  Daniel W. Cranston\thanks{{\it DIMACS, Rutgers University, Piscataway, New Jersey 08854}, \texttt{dcransto@dimacs.rutgers.edu}} \and
  Mickaël Montassier\thanks{{\it LaBRI UMR CNRS 5800, Université Bordeaux I, 33405 Talence Cedex, France}, \texttt{montassi@labri.fr}} \and
  André Raspaud\thanks{{\it LaBRI UMR CNRS 5800, Université Bordeaux I, 33405 Talence Cedex, France}, \texttt{raspaud@labri.fr}} \and
  Weifan Wang\thanks{{\it Departement of Mathematics, Zhejiang Normal University, Jinhua 321004, P. R. China}, \texttt{wwf@zjnu.cn}}
}

\title{Star coloring of sparse graphs}

\newenvironment{proof}{\par \noindent \textbf{Proof.} \\}{\hfill$\Box$}

\newcommand{\ch}{\omega}
\newcommand{\nch}{\omega^*}
\newcommand{\Mad}{\textrm{Mad}}
\newcommand{\degree}{d}

\newtheorem{theorem}{Theorem}
\newtheorem{corollary}[theorem]{Corollary}
\newtheorem{definition}[theorem]{Definition}
\newtheorem{problem}[theorem]{Problem}
\newtheorem{lemma}[theorem]{Lemma}
\newtheorem{observation}[theorem]{Observation}
\newtheorem{remark}[theorem]{Remark}

\date{}

\begin{document}
\maketitle
\renewcommand{\abstractname}{Abstract}
\begin{abstract}
A proper coloring of the vertices of a graph is called a \emph{star
coloring} if the union of every two color classes induces a star forest. The star
chromatic number $\chi_s(G)$ is the smallest number of colors required
to obtain a star coloring of $G$.  In this paper, we study the
relationship between the star chromatic number $\chi_s(G)$ and the
maximum average degree $\Mad(G)$ of a graph $G$. We prove that:
\begin{enumerate}
\item 
If $G$ is a graph with $\Mad(G) < \frac{26}{11}$, then $\chi_s(G)\leq 4$.
\item 
If $G$ is a graph with $\Mad(G) < \frac{18}{7}$ and girth at least 6,
then $\chi_s(G)\leq 5$.
\item 
If $G$ is a graph with $\Mad(G) < \frac{8}{3}$ and girth at least 6,
then $\chi_s(G)\leq 6$.
\end{enumerate}
These results are obtained by proving that such graphs admit a
particular decomposition into a forest and some independent sets.


\end{abstract}

\section{Introduction}

Let $G$ be a graph. 
 A \emph{proper vertex coloring} of $G$ is an assignment $c$ of
integers (or labels) to the vertices of $G$ such that $c(u)\ne c(v)$
if the vertices $u$ and $v$ are adjacent in $G$. A \emph{$k$-coloring}
is a proper vertex coloring using $k$ colors. The minimum $k$ such that
$G$ has a $k$-coloring is called the \emph{chromatic number} of $G$,
denoted by $\chi(G)$.

A proper vertex coloring of a graph is {\it acyclic} if there is no
bicolored cycle in $G$. In other words, the graph induced by the union
of every two color classes is a forest. 
The acyclic coloring of graphs was introduced by Grünbaum in \cite{Gru73}.

Finally, a proper vertex coloring of a graph is called a \emph{star 
coloring} if the union of every two color classes induces a star forest. 
The {\it star chromatic number}, denoted by $\chi_s(G)$, is the smallest 
integer $k$ such that $G$ has a star $k$-coloring.

Star colorings have recently been investigated by Fertin, Raspaud, and
Reed  \cite{FRR01}, Ne\v set\v ril and Ossona de Mendez \cite{NOM03}, and
Albertson \emph{et al.} \cite{AC+04}.

Albertson \emph{et al.} 
proved an interesting relationship between
$\chi_s$, $\chi$, and $\chi_a$ of $G$ and some structures of $G$ (called
\emph{$F$-reductions}), which gives the following corollary.

\begin{corollary}\cite{AC+04}
Let $G$ be a planar graph. 
\begin{enumerate}
\item[$(i)$] If $g(G)\ge 3$, then $\chi_s(G) \le 20$.
\item[$(ii)$] If $g(G)\ge 5$, then $\chi_s(G) \le 16$.
\item[$(iii)$] If $g(G)\ge 7$, then $\chi_s(G) \le 9$.
\end{enumerate}
Here $g(G)$ is the girth of $G$, i.e. the length of a shortest cycle
in $G$.
\end{corollary}
Moreover, Albertson \emph{et al.} gave a planar graph $G$ with $\chi_s(G)=10$ and
they observed that for any girth there exists a planar graph $G$ with
$\chi_s(G)>3$.

\paragraph{}
In this paper, we study the relationship between star coloring and maximum
average degree. The maximum average degree of $G$, denoted by
$\Mad(G)$, is defined as:

$$
\Mad(G) = \max \left\{ \frac{2\cdot|E(H)|}{|V(H)|}, H \subset G\right\}.
$$
Our main result is the following theorem.

\begin{theorem}\label{res}{\ } Let $G$ be a graph.
\begin{enumerate}
\item [$(i)$]
\label{res26/11}
If $\Mad(G) < \frac{26}{11}$, then $\chi_s(G)\leq 4$.

\item [$(ii)$]
\label{res18/7}
If $\Mad(G) < \frac{18}{7}$ and $g(G)\geq 6$, then $\chi_s(G)\leq 5$.

\item [$(iii)$]
\label{res8/3}
If $\Mad(G) < \frac{8}{3}$ and $g(G)\geq 6$, then $\chi_s(G)\leq 6$.
\end{enumerate}
\end{theorem}

It is well-known that if $G$ is a planar graph, then $\Mad(G) < \frac{2\cdot g(G)}{g(G) - 2}$. Hence, we obtain the following corollary.

\begin{corollary}\label{cor} Let $G$ be a planar graph.
\begin{enumerate}
\item [$(i)$]
\label{cor26/11}
 If $g(G)\ge 13$, then $\chi_s(G)\leq 4$.

\item [$(ii)$]
\label{cor5/2}
 If $g(G)\ge 9$, then $\chi_s(G)\leq 5$.

\item [$(iii)$]
\label{cor8/3}
 If $g(G)\ge 8$, then $\chi_s(G)\leq 6$.
\end{enumerate}
\end{corollary}



\begin{remark}
Here we comment that in a previous
submitted version, we proved
that if $\Mad(G) < \frac{5}{2}$ then $\chi_s(G)\leq 5$ (instead of
Theorem \ref{res} $(ii)$). The referee
pointed out that Timmons \cite{Tim07} in his masters
thesis proved several related results, including the following theorem:
\begin{theorem}\cite{Tim07} \label{timmonsRes}
Let $G$ be a planar graph.
\begin{enumerate}
\item [$(i)$] If $g(G)\geq 14$, then $\chi_s(G)\leq 4$
\item [$(ii)$] If $g(G)\geq 9$, then $\chi_s(G)\leq 5$
\item [$(iii)$] If $g(G)\geq 8$, then $\chi_s(G)\leq 6$
\end{enumerate}
\end{theorem}
and sent us the two manuscripts \cite{Tim07,Tim07b}. The two first items
are submitted for publication \cite{Tim07b}. To improve our
results, we have adapted Timmons' proof of Theorem
\ref{timmonsRes} $(ii)$ to have a more general result, which
avoids the planar constraint.
\end{remark}

\vspace{.25cm}
Our proofs are based on an extension of an observation by Albertson
\emph{et al.} \cite{AC+04} (Lemma 5.1). They noticed that if there exists a 
partition of the vertices of a graph $G$ such that: $V(G) = F \dot{\cup} I$ 
and the following two conditions hold:
\begin{enumerate}
\item[P1.] $G[F]$ is a forest 
\item[P2.] $G[I]$ is an independent set such that for all $x,y \in I$,
  $d_G(x,y) > 2$,
\end{enumerate}
then $G$ has a star 4-coloring; we can construct this coloring by star 3-coloring $G[F]$ with color 1, 2, and 3, (see Observation~\ref{3-color forest})
and coloring $G[I]$ with color 4. 
We extend this idea by using several independent sets satisfying
P2 in a certain sense.

\paragraph{}
In Section \ref{sec26/11} we prove Theorem \ref{res} $(i)$, 
in Section \ref{sec18/7} we prove Theorem \ref{res} $(ii)$, and
in Section \ref{sec8/3} we prove Theorem \ref{res} $(iii)$.  
In Section \ref{secConc}, we offer concluding remarks.


\paragraph{Some notation:} 
We use $V(G)$ and $E(G)$ to denote the vertex set and edge set of 
a graph $G$; we use $\delta(G)$ to denote the minimum
degree of $G$. We write $d(v)$ to denote the degree of a vertex $v$.
We call a vertex of degree $k$ (resp.\ at least $k$, at most $k$) a
\emph{$k$-vertex} (resp.\ \emph{$^\ge k$-vertex}, \emph{$^\le
  k$-vertex}). A \emph{$k$-path} is a path of length $k$ (number of
edges). A \emph{$k$-thread} is a $(k+1)$-path whose $k$ internal
vertices are of degree 2.  A \emph{$k_{i_1,\cdots,i_k}$-vertex} with
$i_1\leq \cdots \leq i_k$ is a $k$-vertex that is the initial vertex
of $k$ threads with $i_1, \cdots, i_k$ internal vertices,
respectively.  Usually the threads beginning at a vertex will be disjoint,
but they need not be; for example, each vertex of an isolated $4$-cycle is a 
$2_{3,3}$-vertex. We address this concern explicitly in Claim 4 of Lemma~6.
Let $S\subset V(G)$; we denote by $G[S]$ the subgraph
of $G$ induced by the vertices of $S$.  Let $x$ and $y$ be two
vertices of $G$; we denote by $d_G(x,y)$ the distance in $G$ between
$x$ and $y$, i.e. the length (number of edges) of a shortest path
between $x$ and $y$ in $G$.  In our diagrams, a vertex is colored
black if all of its incident edges are drawn; a vertex is colored
white if it may have more incident edges than those drawn.  We write
$S = S_1\dot{\cup}S_2\dot{\cup}S_3$ to denote that the $S_i$s
partition the set $S$.

\section{Graphs with $\Mad(G) < \frac{26}{11}$\label{sec26/11}}

In this section, we prove that every graph $G$ with $\Mad(G)< \frac{26}{11}$ 
is star 4-colorable.
Lemma~\ref{lemmaStructure26/11} refers to an independent set $I$ such that for all $x,y\in I$, $d_G(x,y) > 2$; we call such a set {\it 2-independent}.

\begin{lemma}\label{lemmaStructure26/11}  
If $ \Mad(G) < \frac{26}{11}$, 
then there exists a vertex partition $V(G)=I\dot{\cup}F$ such that $F$ induces a forest and $I$ is a 2-independent set.
\end{lemma}

\begin{proof}
Let $G$ be a counterexample with the fewest vertices.  We may assume that $G$ is connected and has no vertex of degree 1, since such a vertex could 
be added to $F$.

We begin by proving five claims about structures that must not appear
in $G$.  In each case, let $H$ be the forbidden structure, let
$G'=G-H$, and let $V(G') = I'\dot{\cup}F'$ be the desired partition of
$V(G')$, which is guaranteed by the minimality of $G$.  We conclude
with a discharging argument, which shows that $G$ must contain at
least one of these forbidden structures.  This contradiction implies
that there is no counterexample, and so the lemma is true.

\vspace{1em}

\noindent
\textbf{Claim 1:} $G$ has no $2_{1,1}$-vertex.

Suppose $x$ is a $2_{1,1}$-vertex in $G$.  Consider the partition for the graph obtained by removing $x$ and its neighbors from $G$. We extend this partition to $G$.  If any of the second neighbors of $x$ are in $I'$, then add $x$ and its neighbors to $F'$.  
Otherwise, add $x$ to $I'$ and add the neighbors of $x$ to $F'$.

\vspace{1em}

\noindent
\textbf{Claim 2:} $G$ has no $3_{1,1,2}$-vertex.

Suppose that $x$ is a $3_{1,1,2}$-vertex; let $y$ be an adjacent $2_{0,1}$-vertex.  Consider the partition for the graph obtained by removing $x$, $y$, and their neighbors from $G$.  If $x$ and $y$ lie on a 3-cycle $xyz$, then we add $y$ to $I'$ and add the other vertices to $F'$; hence we suppose that $x$ and $y$ do not lie on a 3-cycle.  If both of the second neighbors 
of $x$ that are not adjacent to $y$ are in $I'$, then add all the removed vertices 
to $F'$.  If both of these second neighbors are in $F'$, then add $x$ to $I'$ and add all other vertices to $F'$. Thus, we may assume that one second neighbor of $x$ is in $F'$, and the other is in $I'$.  If the second neighbor of $y$ not adjacent to $x$ is in $I'$, then add all vertices to $F'$.
Otherwise, add $y$ to $I'$ and all other vertices to $F'$.  

\vspace{1em}

\noindent
\textbf{Claim 3:}  $G$ has no $4_{2,2,2,2}$-vertex.

Suppose that $x$ is a $4_{2,2,2,2}$-vertex in $G$.  Consider the partition for the graph obtained by removing $x$, the neighbors of $x$, and the second neighbors of $x$. To extend the partition to $G$, add $x$ to $I'$ and add the neighbors and second neighbors of $x$ to $F'$.
\vspace{1em}

A 3-vertex that lies on a 3-cycle with two 2-vertices is troublesome for our discharging argument because the adjacent 2-vertices must receive all their charge from the single 3-vertex; hence, we call such a 3-vertex a {\it bad 3-vertex}.

\vspace{1em}
\noindent
\textbf{Claim 4:} $G$ has no $3_{0,1,2}$-vertex adjacent to a bad 3-vertex.

Suppose that $u$ is a $3_{0,1,2}$-vertex adjacent to a bad 3-vertex $v$, which lies on a 3-cycle with 2-vertices $w$ and $x$.  First observe that $u$ must not also lie on a 3-cycle with two 2-vertices; if $u$ does, then the component has only six vertices, so we can easily find a partition into $I$ and $F$. Consider the partition of $G\setminus\{v,w,x\}$.  If $u\not\in I'$, then we add $v$ and $w$ to $F'$ and we add $x$ to $I'$; hence, we assume $u\in I'$.  Note that in $G\setminus\{v,w,x\}$, vertex $u$ is an internal vertex of a 4-thread and is distance at least two from each end-vertex of the 4-thread.  Let $y$ be a neighbor of $u$ on the 4-thread such that $y$ is also distance at least two from each end-vertex.  Consider the partition of $G\setminus\{v,w,x\}$.  If either end-vertex of the 4-thread is in $I'$, then we can remove $u$ from $I'$, but otherwise we can add $u$ to $F'$ and add $y$ to $I'$; now we proceed as above.

\vspace{1em}

We now show that $G$ cannot have certain types of cycles.  
We want $J$ to be the maximum subgraph of $G$ such that all its vertices either are $2_{0,1}$-vertices or are $3_{0,1,2}$-vertices; further, if $d_G(v)=3$, then we want $d_J(v)\geq 2$.
To form $J$, let $A$ be the set of $2_{0,1}$-vertices and of $3_{0,2,2}$-vertices.  Add to $A$ all $3_{0,1,2}$-vertices that are adjacent to other $3_{0,1,2}$-vertices.  Let $J$ be the subgraph induced by $A$.

\vspace{1em}

\noindent
\textbf{Claim 5:}  
Every component of $J$ is a tree or a cycle.

Suppose $J$ contains a cycle with vertex set $C$ that contains a vertex $r$ with $\degree_{J}(r)=3$.  Choose such a cycle so that $|C|$  is minimal.  
By minimality $C$ induces a chordless cycle in $G$.

Let $V(G-C) = F'\dot{\cup}I'$ be a vertex partition of $G - C$, where we may assume that a 1-vertex in $G-C$ or
a 2-vertex adjacent to a 1-vertex is in $F'$. We extend this partition to all of $G$ in two phases.  
In phase 1, we will form $I''$ and $F''$ by successively adding the vertices of $C$ to $I'$ and $F'$ so that $I''$ is a 2-independent set in $G$, and 
the only possible cycle induced by $F''$ is $C$ itself.  
In phase 2, we modify $I''$ and $F''$, if necessary, so that 
$I'''$ is a 2-independent set and $F'''$ is a acyclic.

Before beginning phase 1, we also need two definitions.
For convenience, we assign a cyclic orientation to $C$. If vertices $u$ and $v$ are in 
$C \cap F''$, then $u$ is a \textit{predecessor} of $v$ if the directed path along $C$ from 
$u$ to $v$ is entirely in $F''$.  We say that $F''$ is {\it well-behaved}, if for all
vertices $w,v$ with $w\in (F'' - C)$, $v\in (F'' \cap C)$, and $vw \in E(G)$, 
we have either that $w$ is a leaf in $F''$ or that $v$ has at most two predecessors, each of which is a 2-vertex in $G$. 
We now show why it is useful to have $F''$ well-behaved.
Aiming for a contradiction, assume $B$ is a cycle in $F''$ in which $B \neq C$.  Let $u$ and $v$ be the endpoints of a maximal component of 
$B \cap C$.  Without loss of generality, assume $u$ is a predecessor of $v$.  
Let $w$ be the neighbor of $v$ not on $C$.  Since $w$ lies on $B$, 
$w$ is not a leaf in $F''$.  Since $F''$ is well-behaved, the two predecessors of $v$ have degree 2 
in $G$.  This contradicts the fact that $\degree_{G}(u)=3$.  
Observe that at the start of phase~1 we have $F''=F'$ and $I''=I'$ so that
$F''$ is trivially well-behaved. We will ensure that $F''$ remains well-behaved throughout phase~1, so that at the end of the phase the only possible cycle induced by $F''$ is $C$.

 \begin{figure}[htbp]
 \begin{center}
 \psfrag{w'}{{\normalsize$w_1$}}
 \psfrag{v'}{{\normalsize$v_1$}}
 \psfrag{s}{{\normalsize$s$}}
 \psfrag{t}{{\normalsize$t$}}
 \psfrag{u}{{\normalsize$u$}}
 \psfrag{v}{{\normalsize$v$}}
 \psfrag{w}{{\normalsize$w$}}
 \psfrag{x}{{\normalsize$x$}}
 \psfrag{y}{{\normalsize$y$}}
 \includegraphics[width=7cm]{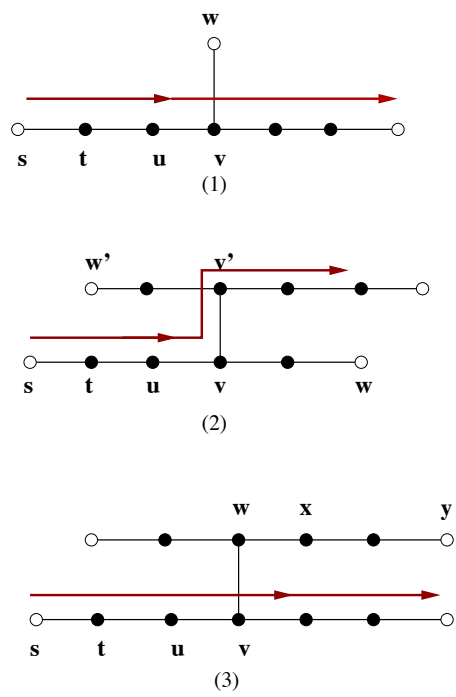}

 \caption{The three key configurations of Claim~5.} \label{fig18/7}
 \end{center}
 \end{figure}

To execute phase 1, observe that the immediate neighborhood (in $G$) of each 
$3_{0,1,1}$-vertex $v$ on $C$ must appear as in Figure~1.1 or Figure~1.2, where the 
arrows indicate $C$; recall that the black vertices in the figures have all neighbors 
shown, but white vertices may have additional neighbors (in $G$). One by one, 
we examine the neighborhood of each 3-vertex 
on $C$ proceeding in the order as we encounter them on $C$, and assign each vertex either to $I''$ or $F''$. If $C$ contains a vertex $v$ with a 3-neighborhood as in Figure~1.2, 
then we begin at such a $v$.  
Otherwise we choose a 3-vertex $v$ as in Figure~1.1 for which, if possible, $w\in I''$.

Suppose first that the 3-neighborhood of $v$ looks like in Figure~1.1.  
If $s$ or $w$ is in $I''$, then add $t,u,v$ to $F''$.  
If $s$ and $w$ are not in $I''$, then add $u$ to $I''$
and $t,v$ to $F''$. This keeps $F''$ well-behaved: 
If $u$ is assigned to $I''$, then $v$ has no predecessors.  If $u$ is not assigned to $I''$, 
then it must be the case that $s$ or $w$ is in $I''$.  If $s\in I''$, then $v$ has two predecessors, both of which are 2-vertices in $G$, whereas if $w\in I''$, 
then $F''$ is trivially well-behaved. 
Suppose then that the neighborhood of $v$ looks like in Figure~1.2.  
If $w\in F''$, then add $v$ to $I''$ and $t,u,v_1$ to $F''$.
If $w\in I''$, and $w_1\in F''$, then add $v_1$ to $I''$, and 
 $t,u,v$ to $F''$. If $w,w_1\in I''$, then add $t,u,v,v_1$ 
 to $F''$. To see that $F''$ is well-behaved, observe that when $w\in F''$, then $v$ is assigned to $I''$, so $v\notin F''$ and $v_1$ has no predecessors.  
 If $w\in I''$, then the neighbor of $v$ not in $C$ is a leaf of $F''$. Furthermore, if also
 $w_1\in I''$, then the neighbor of $v_1$ not in $C$ is a leaf in $F''$, whereas if $w_1\in F''$, then $v_1\notin F''$. This completes phase 1. 

If $F''$ is acyclic, then we have extended the partition to $G$ so we may assume that
$F''$ contains a cycle. Since $F''$ is well-behaved, this cycle must be $C$. We correct this problem in phase 2. 
All of $C$ lies in $F''$ precisely when for each 3-vertex on $C$ as in Figure~1.1 we have 
$w\in I''$, and for each 3-vertex as in Figure~1.2 we have $w,w_1\in I''$.  
To form $I'''$ and $F'''$, recall that $C$ contains a vertex $r$ with $\degree_{J}(r)=3$; 
that is $r$ is adjacent to a $3_{0,1,1}$-vertex and to two $2_{0,1}$-vertices. If $C$ runs through $v=r$ as 
in Figure~1.1, then the  immediate neighborhood of $v$ must look like Figure~1.3.  
In this case we must have $w\in I''$. If also $y\in I''$, then let $I'''=I''-w+u$, whereas if $y\in F''$, then let $I'''=I''-w+u+x$. In the remaining case  $C$ runs through $r$ as 
in Figure~1.2. If $r=v$, then $w$ is a 2-vertex with $w\in I''$, which contradicts our assumption that degree 2 vertices adjacent to leaves in $G-C$ are in $F'$, and thus $F''$. If $r=v_1$,
then we get a similar contradiction for $w_1$. In any case, 
$I'''$ is 2-independent and $F'''=V(G)-I'''$ is now acyclic, finishing the proof of Claim~5.

\vspace{1em}

\paragraph{}

Now we prove that every graph $G$ with minimum degree 2 and $\Mad(G) < \frac{26}{11}$ contains one of the configurations forbidden by Claims 1--5.  
Let $G$ be a counterexample, i.e. a graph with $\delta(G)\geq2$, with $\Mad(G) < \frac{26}{11}$, and containing none of the configurations forbidden by Claims 1--5.

We use a discharging argument with initial charge $\ch(v) = \degree(v)$ at each vertex $v$ and with the following four discharging rules (R1) -- (R4),
where $A$ and $J$ refer to the vertex set and subgraph from Claim~5.  
We write $\nch(v)$ to denote the change at each vertex $v$ after we apply the discharging rules.
If $v$ is a $^{\geq}3$-vertex, then a 2-vertex $w$ is {\it nearby} $v$ if $v$ is adjacent to $w$, or if 
$v$ and $w$ have a common neighbor of degree 2 (in $G$).  

Note that the discharging rules do not change the sum of the charges.
To complete the proof, we show that $\nch(v)\ge
\frac{26}{11}$ for all $v\in V(G)$; this leads to the following obvious
contradiction:
 $$
 \frac{26}{11}\leq \frac{\frac{26}{11}|V(G)|}{|V(G)|} \leq \frac{\sum_{v\in
 V(G)}\nch(v)}{|V(G)|} =  \frac{\sum_{v\in
 V(G)}\ch(v)}{|V(G)|} = \frac{2|E(G)|}{|V(G)|} \le \Mad(G) < \frac{26}{11}.
 $$
Hence no counterexample can exist.

\paragraph{}
The discharging rules are defined as follows. 

\begin{enumerate}

\item[(R1)] Each $^{\geq}3$-vertex gives $\frac{2}{11}$ to each nearby 2-vertex;
each bad 3-vertex gives each nearby 2-vertex an additional $\frac2{11}$, for a total of $\frac4{11}$.

\item[(R2)] Each $3_{0,1,1}$-vertex receives $\frac{1}{11}$ from its neighbor of degree at least 3, unless this neighbor is a bad 3-vertex or is in $A$, in which case no charge is transferred.

\item[(R3)] 
If 
$\degree_{J}(v)=3$, then $v$ gets $\frac{1}{11}$ from the bank.

\item[(R4)] If 
$v\notin A$ and $v$ is adjacent to $k$ $2_{0,1}$-vertices, then $v$ gives $\frac{k}{11}$ to the bank.

\end{enumerate}

We now verify that $\nch(v)\geq \frac{26}{11}$ for each $v\in V(G)$.  
Note that each leaf in $J$ is a $2_{0,1}$-vertex.  Hence, by (R4), the bank receives $\frac1{11}$ for each leaf of $J$.  By (R3), the bank gives away $\frac1{11}$ for each 3-vertex of $J$.
Claim~5 implies that $J$ has average degree at most 2, so $J$ contains at least as many leaves as 3-vertices; hence, the bank has nonnegative charge.

\begin{enumerate}

\item [\emph{Case $d(v)=2$}] By (R1), $v$ receives $\frac2{11}$ from each nearby $^{\geq}3$-vertex, so $\nch(v) = 2 + 2 \cdot \frac{2}{11} = \frac{26}{11}$; if $v$ is adjacent to a bad 3-vertex, then $\nch(v) = 2 + \frac{4}{11} = \frac{26}{11}$.

\item [\emph{Case $d(v)=3$}]
If $v$ is a 3-vertex, then we consider several possibilities:  

If $v$ is a $3_{1,1,1}$-vertex, then by Claim~2, $v$ is not adjacent to a $2_{0,1}$-vertex.  
Hence, $v$ does not send any charge to the bank.  In this case, 
$\nch(v) = 3 - 3 \cdot \frac{2}{11}  = \frac{27}{11}$.  

If $v$ is a bad 3-vertex, then $\nch(v) = 3 - 2\cdot\frac{4}{11} + \frac{1}{11} = \frac{26}{11}$.

If $v$ is a $3_{0,2,2}$-vertex, then $v\in A$.  Let $w$ be the neighbor of $v$ that is not a 2-vertex.  If $vw \in E(J)$, then $\degree_{J}(v)=3$.  
By (R3), $v$ will get $\frac{1}{11}$ from the bank.  By (R1), $v$ sends 
a total of $4 \cdot \frac{2}{11} $ to nearby 2-vertices.  Thus $\nch(v) = 3 + \frac{1}{11} - 4 \cdot \frac{2}{11}  = \frac{26}{11}$.
If $vw \notin E(J)$, then $w\notin A$, so that by (R2), $v$ receives $\frac{1}{11}$ 
from $w$.  Therefore the same calculation applies to $v$.  

If $v$ is a $3_{0,1,2}$-vertex (but not a $3_{0,2,2}$-vertex), then 
$v$ gives $\frac{1}{11}$ to the bank.  Note from Claim 4 that $v$ does not give charge to a bad 3-vertex.  Therefore $\nch(v) \geq 3 - 3 \cdot \frac{2}{11}  - \frac{1}{11} = \frac{26}{11}$.

If $v$ is a $3_{0,1,1}$-vertex (but not a $3_{0,1,2}$-vertex), then $v$ may give charge $\frac{1}{11}$ to a bad 3-vertex.  So $\nch(v)\geq 3 - 2\cdot \frac{2}{11} - \frac{1}{11} = \frac{28}{11}$.

If $v$ is a $3$-vertex adjacent to at most one 2-vertex, then $\nch(v) \geq 3 - 2 \cdot \frac{2}{11}  - 2 \cdot \frac{1}{11}  - \frac{1}{11} = \frac{26}{11}$.

\item [\emph{Case $d(v)=4$}]
If $v$ is a 4-vertex, then by Claim~3, $v$ is adjacent to at most three $2_{0,1}$-vertices, so $\nch(v) \geq 4 - 7 \cdot \frac{2}{11}  - 3 \cdot \frac{1}{11}  = \frac{27}{11}$.  

\item [\emph{Case $d(v)\geq5$}]
If $v$ is a $^{\geq}5$-vertex, then $\nch(v) \geq \degree(v)  - \frac{4}{11} \degree(v) - \frac{1}{11} \degree(v) = \frac{6}{11} \degree(v) \geq \frac{30}{11}$.
\end{enumerate}

This implies that $\Mad(G)$ $\geq \frac{26}{11}$, which provides the needed contradiction.
\end{proof}       
\begin{observation}
\label{3-color forest}
If $F$ is a forest, then $\chi_s(G)\leq 3$.
\end{observation}
In each component $T$ of $F$, arbitrarily choose a root $r$; color each vertex $v$ of $T$ with color $(1 + d_T(r,v))\pmod 3$.
\paragraph{}{\ } 

\par \noindent \textbf{Proof of Theorem \ref{res} $(i)$} \\ Let $G$ be a
graph with $\Mad(G)<\frac{26}{11}$. By Lemma
\ref{lemmaStructure26/11}, there exists a vertex partition
$V(G)=F\dot{\cup} I$, where $F$ induces a forest and $I$ is a
2-independent set.
 
Now we color:
\begin{itemize}
\item $F$ with colors 1, 2, and 3 (as in Observation~\ref{3-color forest}), and
\item $I$ with color 4. 
\end{itemize}
This produces a star 4-coloring of $G$.  \hfill{$\Box$}

\section{Graphs with $\Mad(G) < \frac{18}{7}$ and $g(G)\geq 6$ \label{sec18/7}}

In this section, we prove that every graph $G$ with $\Mad(G)<
\frac{18}{7}$ and $g(G)\geq 6$ is star 5-colorable.  The proofs in
this section is based on Timmons' proofs \cite{Tim07,Tim07b}.

\bigskip

\noindent The proofs in this section and the next section closely resemble thos
in Section 2. The main difference is that in both this section and the
next, we split the argument inti two lemmas, rather than one.

\bigskip

\noindent In this section, we first prove that each graph of interest
must contain one of 14 configurations; we use this lemma to prove a
structural decomposition result, which in turns leads to the star
coloring result. The outline of Section~\ref{sec8/3} is similar.

\begin{lemma}\label{lemmaStructure18/7}
 A graph $G$ with $\Mad(G) < \frac{18}{7}$ 
contains one of the following 14 
 configurations:
 \begin{enumerate}
 \item A $^\leq 1$-vertex.
 \item A $2_{1,1}$-vertex.
 \item A $3_{1,1,1}$-vertex.
 \item A $3_{0,0,2}$-vertex.
 \item A $3_{0,1,1}$-vertex adjacent to a $3_{0,1,1}$-vertex. 
 \item A $3_{0,0,1}$-vertex adjacent to two $3_{0,1,1}$-vertex.
 \item A $4_{1,1,1,2}$-vertex.
 \item A $4_{0,2,2,2}$-vertex adjacent to a 3-vertex.
 \item A $4_{0,1,1,1}$-vertex adjacent to a $3_{0,1,1}$-vertex.
 \item A $4_{0,1,1,1}$-vertex adjacent to a $4_{0,2,2,2}$-vertex.
 \item A $4_{0,0,2,2}$-vertex adjacent to two $4_{0,2,2,2}$-vertex.
 \item A $4_{0,0,2,2}$-vertex adjacent to a $4_{0,2,2,2}$-vertex and a $3_{0,1,1}$-vertex.
 \item A $5_{1,2,2,2,2}$-vertex.
 \item A $5_{0,2,2,2,2}$-vertex adjacent to a $4_{0,2,2,2}$-vertex.


 \end{enumerate}
 \end{lemma}

\begin{proof}
 Let $G$ be a counterexample, i.e.\ a graph with $\Mad(G) < \frac{18}7$, with $g(G)\geq 6$, and containing none of the
 configurations of Lemma \ref{lemmaStructure18/7}. 

We use a discharging argument with initial charge $\ch(v)=d(v)$ at each vertex vertex $v$, and with the following four discharging rules (R1) -- (R4), 
which describe how to redistribute the charge.
We write $\nch(v)$ to denote the charge at each vertex $v$ after we apply the discharging rules. 
Note that the discharging rules do not change the sum of the charges.
To complete the proof, we show that $\nch(v)\ge
\frac{18}{7}$ for all $v\in V(G)$; this leads to the following obvious
contradiction:
 $$
 \frac{18}{7}\leq \frac{\frac{18}{7}|V(G)|}{|V(G)|} \leq \frac{\sum_{v\in
 V(G)}\nch(v)}{|V(G)|} =  \frac{\sum_{v\in
 V(G)}\ch(v)}{|V(G)|} = \frac{2|E(G)|}{|V(G)|} \le \Mad(G) < \frac{18}{7}.
 $$
Hence no counterexample can exist.

\paragraph{}
The discharging rules are defined as follows. 

Each $^\geq 3$-vertex gives: 
 \begin{enumerate}
 \item[(R1)] $\frac{2}{7}$ to each 2-vertex that is not adjacent to
   a 2-vertex.
 \item[(R2)] $\frac{4}{7}$ to each 2-vertex $u$ that is adjacent to
   a 2-vertex ($u$ is a $2_{0,1}$-vertex).
 \item[(R3)] $\frac{1}{7}$ to each $3_{0,1,1}$-vertex.
 \item[(R4)] $\frac{2}{7}$ to each $4_{0,2,2,2}$-vertex.


 \end{enumerate}

To complete the proof, it suffices to verify that 
$\nch(v) \geq\frac{18}{7}$ for all $v\in V(G)$.

\begin{enumerate}

\item [\emph{Case $d(v)=2$}] If $v$ is a $2_{0,1}$-vertex, then $v$
  receives $\frac{4}{7}$ from its neighbor of degree at least
  3. Hence, $\nch(v) = 2 + \frac{4}{7}=\frac{18}{7}$. Otherwise, $v$
  is adjacent to two $^\geq 3$-vertices, which each give $\frac{2}{7}$
  to $v$. Hence $\nch(v) = 2 +2\cdot  \frac{2}{7}=\frac{18}{7}$.

\item [\emph{Case $d(v)=3$}] By Lemma~\ref{lemmaStructure18/7}.4, 
  $v$ is not incident to a
  2-thread. 

  If $v$ is a $3_{0,1,1}$-vertex, then $v$ gives
  $\frac{2}{7}$ to each adjacent 2-vertex and receives $\frac{1}{7}$
  from its third neighbor (since this third neighbor is neither a
  2-vertex, nor a $3_{0,1,1}$-vertex, nor a $4_{0,2,2,2}$-vertex). Hence,
  $\nch(v) = 3 - 2\cdot \frac{2}{7} + \frac{1}{7}= \frac{18}{7}$.

  If $v$ is a $3_{0,0,1}$-vertex, then $v$ is adjacent to
  at most one $3_{0,1,1}$-vertex and to no
  $4_{0,2,2,2}$-vertex. Hence, $\nch(v) \geq 3 - \frac{2}{7} -
  \frac{1}{7} = \frac{18}{7}$.

  Finally, assume that $v$ is not adjacent to a 2-vertex. Since
  $v$ is not adjacent to a $4_{0,2,2,2}$-vertex, $\nch(v)\geq 3
  - 3\cdot \frac{1}{7} = \frac{18}{7}$. 

\item [\emph{Case $d(v)=4$}] By Lemma~\ref{lemmaStructure18/7}.8, 
  $v$ is incident to at most three
  2-threads. 

  If $v$ is incident to exactly three 2-threads, then the fourth
  neighbor of $v$ is neither a 2-vertex, nor a $3_{0,1,1}$-vertex,
  nor a $4_{0,2,2,2}$-vertex; hence, this fourth neighbor gives $\frac{2}{7}$ to $v$. 
  Thus $\nch(v)= 4- 3\cdot \frac{4}{7} + \frac{2}{7} = \frac{18}{7}$.

  Assume that $v$ is incident to exactly two 2-threads. 
  If the third (or fourth) neighbor of $v$ is a 2-vertex or a $4_{0,2,2,2}$-vertex, then the final neighbor of $v$ does not receive charge from $v$ (by Lemmas~\ref{lemmaStructure18/7}.7~and~\ref{lemmaStructure18/7}.9--\ref{lemmaStructure18/7}.12).  In this case, 
  $\nch(v) = 4 - 2 \cdot \frac{4}{7} - \frac{2}{7} = \frac{18}{7}$. 
  It is also possible that one or both of the remaining neighbors of $v$ are $3_{0,1,1}$-vertices.
In this case, $\nch(v) \geq 4 - 2 \cdot \frac{4}{7} - 2 \cdot \frac{1}{7} = \frac{18}{7}$.

  If $v$ is incident to at most one 2-thread, then $\nch(v)
  \geq 4 - \frac{4}{7} - 3\cdot \frac{2}{7} = \frac{18}{7}$.

\item [\emph{Case $d(v)= 5$}] By Lemma~\ref{lemmaStructure18/7}.13, $v$ is incident to at most four 2-threads. 

  If $v$ is incident to exactly four 2-threads, then
  its fifth neighbor is neither a 2-vertex, nor a
  $4_{0,2,2,2}$-vertex. Hence, $\nch(v) \geq  5 - 4\cdot\frac{4}{7} -
  \frac{1}{7}= \frac{18}{7}$. 

  If $v$ is incident to at most three 2-threads, then
  $\nch(v) \geq 5 - 3\cdot \frac{4}{7} - 2\cdot \frac{2}{7} = \frac{19}{7}$.

\item [\emph{Case $d(v)\ge 6$}] Trivially, $\nch(v) \geq d(v) - d(v)\cdot
  \frac{4}{7} = \frac{3}{7}d(v) \geq \frac{18}{7}$.

\end{enumerate}

This implies that $\Mad(G)$ $\geq \frac{18}{7}$, which provides the needed contradiction.
\end{proof}

\begin{lemma}\label{structure18/7}
If $G$ is a graph with  $\Mad(G) < \frac{18}{7}$ and $g(G)\ge 6$, then there exists a vertex partition $V(G) = F \dot{\cup} I_1 \dot{\cup} I_2$ such that:
\begin{enumerate}
\item[P1.] $F$ induces a forest, 
\item[P2.] $I_1$ is an independent set such that
  for all $x,y \in I_1$, $d_{G[F\cup I_1]}(x,y)> 2$, and
\item[P3.] $I_2$  is an independent set such that
  for all $x,y \in I_2$, $d_{G}(x,y)> 2$.
\end{enumerate}
\end{lemma}

 \begin{figure}[htbp]
 \begin{center}
\hspace{-2.5cm}
 \includegraphics[width=12cm]{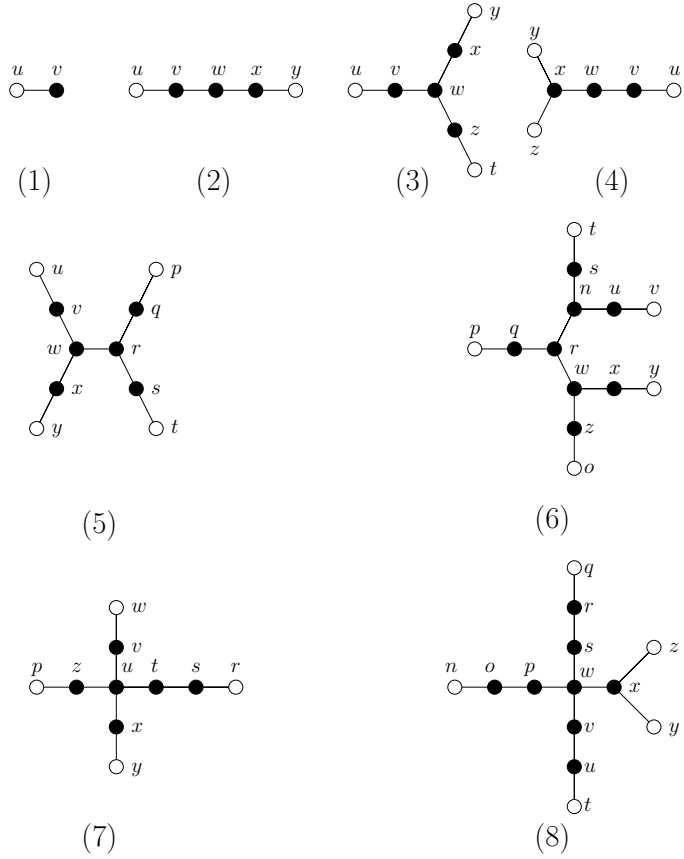}
 \caption{(1 of 2) The first 8 of the 14 unavoidable configurations in Lemma~\ref{lemmaStructure18/7}.} \label{fig18/7}
 \end{center}
 \end{figure}

\begin{figure}[htbp]
 \begin{center}
\hspace{-2.5cm}
 \includegraphics[width=12cm]{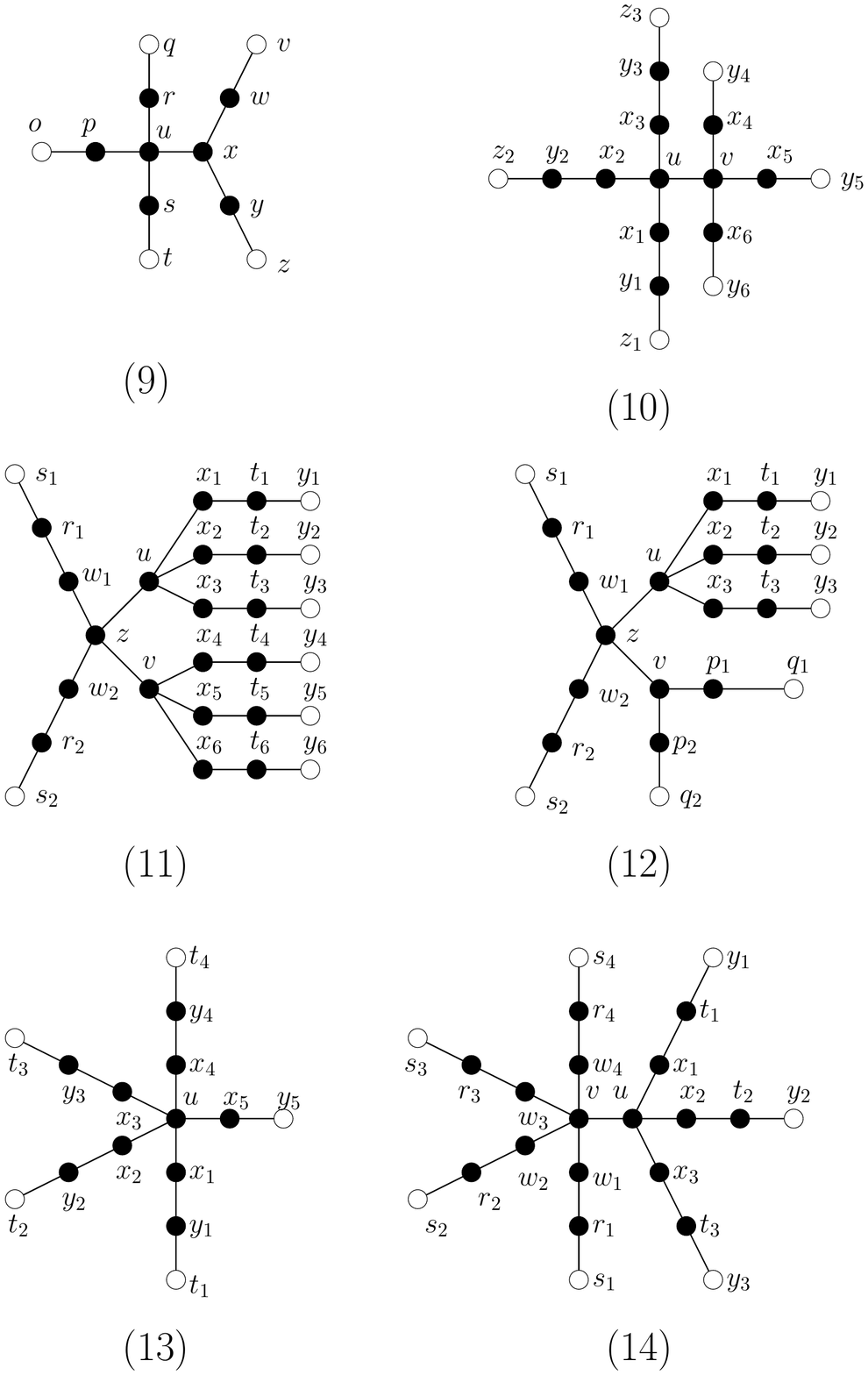}
 \caption{(2 of 2) The last 6 of the 14 unavoidable configurations in Lemma~\ref{lemmaStructure18/7}.} \label{fig18/7bis}
 \end{center}
 \end{figure}

\begin{proof}
Let $G$ be a counterexample with the fewest number of vertices. By Lemma
\ref{lemmaStructure18/7},  $G$ must contain one of the 14
configurations. Using the fact that $g\geq 6$ and we have
$2_{1,1}$-vertex being forbidden, it follows that these configurations
look like in Figures~\ref{fig18/7} and~\ref{fig18/7bis} (where some of
the white vertices may represent the same vertex). 

In each configuration,
let $N$ denote the set of black vertices. 
In each case, we 
delete a subgraph $H$ and partition $V(G-H)$ into sets $F'$, $I_1'$, and $I_2'$ that satisfy $P_1$, $P_2$, and $P_3$.
We now show how to add the deleted vertices to $F'$, $I_1'$, and $I_2'$ so that the resulting sets still satisfy $P_1$, $P_2$, and $P_3$.

We consider each of the 14 configurations in Figures~\ref{fig18/7} and~\ref{fig18/7bis}.
\begin{enumerate}
\item [(1)] 
Let $G' = G-v$. By the minimality of $G$, the vertex set of 
$G'$ is the disjoint union of sets $F'$, $I_1'$, and $I_2'$  that satisfy P1, P2, and P3.
Now add $v$ to $F'$.

\item [(2)] 
Let $G'= G\setminus N$.
  If $u,y \in F'$, then add $v$ and $x$ to $F'$ and add $w$ to $I_1'$.
Otherwise, 
add $v$, $w$, $x$ to $F'$.

\item [(3)] 
Let $G'=G\setminus N$.
  If $u,y,t \in F'$, then 
  add $v$, $x$, $z$ to $F'$ and add $w$ to $I_1'$.
If exactly two of the vertices $u$, $y$, $t$ are in $F'$, then w.l.o.g.\
  $u,y\in F'$ and $t \in I_1$ (resp. $I_2$); 
add $v$, $x$, $z$ to $F'$ and add $w$ to $I_2'$ (resp. $I_1'$).
If at most one of the vertices $u$, $y$, $t$ is in $F'$, then 
add $v$, $w$, $x$, $z$ to $F'$.
\item [(4)] 
Let $G' = G\setminus\{v,w\}$.
Consider the case $u,x \in F'$. If
$y,z\in I_1'\cup I_2'$
then
add $v$ and $w$ to $F'$.
Otherwise: 
if $y,z\in I_2'\cup F'$, then add $v$ to $F'$ and $w$ to $I_1'$;
similarly, 
if $y,z\in I_1'\cup F'$, then add $v$ to $F'$ and add $w$ to $I_2'$.
Finally, if $u\notin F'$ or $x\notin F'$, then 
add $v$ and $w$ to $F'$.

\item [(5)] 
Let $G'=G\setminus N$. If $u,p,y,t\in F'$, then 
add $v$, $q$, $x$, $s$ to $F'$, add $w$ to $I_1'$, and add $r$ to $I_2'$.
If exactly three of $u$, $p$, $y$, $t$ are in $F'$, then W.l.o.g., $u,p,y\in F'$ and $t\in I_1'$; add $w$ to $I_1'$ and $v$, $p$, $r$, $s$, $x$ to $F'$.
The remaining cases are easier and are left to the reader.

\item [(6)(7)(9)] These easy cases are left to the reader.

\item [(8)] Let $G'=G\setminus \{o,p,r,s,u,v,w\}$. We consider the
  different cases according to $x$. If $x$ belongs to $I_2$, we put
  $w$ in $I_1$ and the remaining vertices in $F$. Suppose now that $x$
  is in $F$. If an independent set $I_i$ does not appear for $x$ and
  $y$, we put $w$ in $I_i$ and the remaining vertices in $F$. So
  assume that w.l.o.g. $y$ is in $I_1$ and $z$ is in $I_2$. We put $s$
  and $p$ if necessary in a different independent set according to $q$
  and $n$, and the remaining vertices in $F$. Finally suppose that $x$
  is in $I_1$. If $y$ and $z$ are not in $I_2$, we put $w$ in $I_2$
  and the remaining vertices in $F$. Otherwise, we change $x$ and put
  it in $F$ and obtain a previous case.

\item [(10)] 
Let $G'=G\setminus N$. We consider
  two cases: (i) If at least two of $y_4,y_5,y_6$ are in independent
  sets, then add $u$ to an independent set and add the other vertices
  to the forest; (ii) If at most one of $y_4,y_5,y_6$ is in an
  independent set, then add $u$ to that set and $v$ to the other 
  independent set and add all the other vertices to the forest.

\item [(11)] 
Let $G'=G\setminus N$. Add $u$ and $v$ to
  $I_1$, add $z$ to $I_2$,  and add all the other vertices to the forest.

\item [(12)] 
Let $G'=G\setminus N$. We just
  consider the case where $q_1$ and $q_2$ are in $F'$; the remaining
  cases are easier. Add $u$ and $v$ to $I_1$, add $z$ to $I_2$, and add the
  other vertices to the forest.

\item [(13)] 
Let $G'=G\setminus N$. 
  Add $u$ to an independent set that doesn't contain $y_5$ and add all the other vertices to the forest.

\item [(14)] 
Let $G'=G\setminus N$. 
Add $u$ to $I_1'$, add $v$ to $I_2'$, and add all other other vertices to the forest.

\end{enumerate}

\end{proof}

\par \noindent \textbf{Proof of Theorem \ref{res} $(ii)$} \\ Let $G$ be a
graph with $\Mad(G)<\frac{18}{7}$ and $g(G)\ge 6$. By Lemma
\ref{structure18/7}, there exists a vertex partition $V(G)=F\dot{\cup} I_1\dot{\cup} I_2$. Now
we color: 
\begin{itemize} 
\item $G[F]$ with the colors 1, 2, and 3 (as in Observation~\ref{3-color forest}),
\item $G[I_1]$ with the color 4, and
\item $G[I_2]$ with the color 5.

\end{itemize} 
This produces a star 5-coloring of $G$.  \hfill{$\Box$}

\begin{remark}
The condition $g(G)\ge 6$ can be dropped in Lemma \ref{structure18/7}
and so in Theorem \ref{res} $(ii)$. Since there are no 3-threads,
we just use this assumption to ensure that white and black vertices
are distinct in the configurations in Figures \ref{fig18/7} and \ref{fig18/7bis}.
\end{remark}

\section{Graphs with $\Mad(G) < \frac{8}{3}$ and $g(G)\ge 6$\label{sec8/3} }

In this section, we prove that every graph $G$ with $\Mad(G)<
\frac{8}{3}$ and $g(G)\geq 6$ is star 6-colorable.  We first prove a
structural lemmma. This lemma is then used to prove the existence of a
particular decomposition which in turns leads to the star coloring
result.

\begin{lemma}\label{lemmaStructure8/3}
A graph $G$ with $\Mad(G) < \frac{8}{3}$ and $g(G)\ge 6$ contains one of the following 8 configurations:
\begin{enumerate}
\item A $^\leq 1$-vertex.
\item A $2_{0,1}$-vertex adjacent to only $^\le 4$-vertices. 
\item A $3_{0,1,1}$-vertex adjacent to only  $^\le 3$-vertices.
\item A $k_{1,1,2,\cdots,2}$-vertex with $k\geq 5$.
\item A $k_{0,2,\cdots,2}$-vertex adjacent to a $3$-vertex with $k\geq 5$.

\item A $k_{0,1,2,\cdots,2}$-vertex adjacent to a 
  $k'_{0,1,1,2,\cdots,2}$-vertex with $k,k'\geq 5$.

\item A $k_{0,0,2,\cdots,2}$-vertex adjacent to two 
  $3_{0,1,1}$-vertices with $k\geq 5$.

\item A $k_{0,0,2,\cdots,2}$-vertex adjacent to a
  $k'_{0,1,1,2,\cdots,2}$-vertex and to a $k''_{0,1,2,2,\cdots,2}$-vertex with $k,k',k''\geq 5$.

\end{enumerate}
\end{lemma}

\begin{figure}[htbp]
\begin{center}
\includegraphics[width=11.0cm]{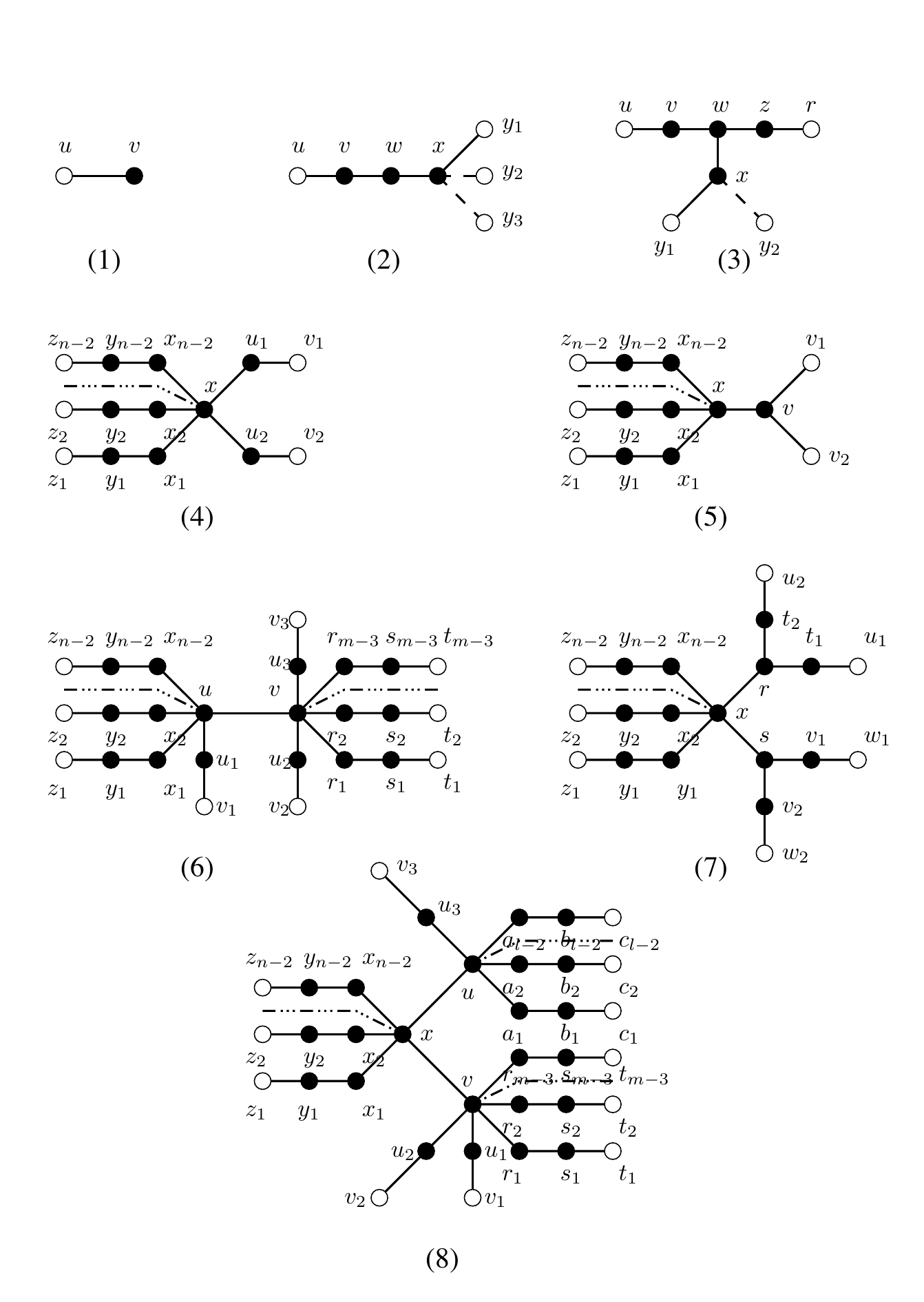}
\caption{The 8 unavoidable configurations in Lemma~\ref{lemmaStructure8/3}.} \label{fig8/3}
\end{center}
\end{figure}

\begin{proof}
Our proof is similar to the proofs of Lemmas \ref{lemmaStructure26/11} and
\ref{lemmaStructure18/7}.
Let $G$ be a counterexample. 

We use a discharging argument with initial charge $\ch(v)=d(v)$ at each vertex vertex $v$, and with the following three discharging rules (R1) -- (R3), 
which describe how to redistribute the charge.
We write $\nch(v)$ to denote the charge at each vertex $v$ after we apply the discharging rules. 
Note that the discharging rules do not change the sum of the charges.
To complete the proof, we show that $\nch(v)\ge
\frac{8}{3}$ for each $v\in V(G)$; this leads to the following obvious
contradiction:

$$
\frac{8}{3}\leq \frac{\frac{8}{3}|V(G)|}{|V(G)|} \leq \frac{\sum_{v\in
V(G)}\nch(v)}{|V(G)|} =  \frac{\sum_{v\in
V(G)}\ch(v)}{|V(G)|} = \frac{2|E(G)|}{|V(G)|} \le \Mad(G) < \frac{8}{3}.
$$
Hence no counterexample can exist.

\paragraph{}
The discharging rules are defined as follows.

\begin{itemize}
\item[(R1)] Each $^\ge 5$-vertex gives $\frac{2}{3}$ to each adjacent
  2-vertex $u$ which is adjacent to a 2-vertex ($u$ is a
  $2_{0,1}$-vertex).
\item[(R2)] Each $^\ge 3$-vertex gives $\frac{1}{3}$ to each adjacent
  2-vertex which is not adjacent to a 2-vertex. 
\item[(R3)] Each $^\ge 4$-vertex gives $\frac{1}{3}$ to each adjacent
  $3_{0,1,1}$-vertex and $\frac13$ to each adjacent $5_{0,2,2,2,2}$-vertex. 
\end{itemize}

We call $3_{0,1,1}$-vertices and $5_{0,2,2,2,2}$-vertices {\it light vertices}.  
If a vertex is the initial or final vertex of a thread, we call it an {\it end} of the thread.
By Lemma \ref{lemmaStructure8/3}.1., $\delta(G)\geq 2$.
\begin{enumerate}
\item [\emph{Case $d(v)=2$}]  If $v$ is not adjacent to any 2-vertices, then it
  receives $\frac{1}{3}$ from each neighbor by (R2); so, $\nch(v) = 2 +
  2\cdot \frac{1}{3} = \frac{8}{3}$. If $v$ is a $2_{0,1}$-vertex,
  then it is adjacent to a $^\ge 5$-vertex. By (R3), $u$ gives 
  $\frac{2}{3}$ to $v$, so $\nch(v) = 2 + \frac{2}{3} = \frac{8}{3}$.

\item [\emph{Case $d(v)=3$}]  By Lemma \ref{lemmaStructure8/3}.2 and
  the girth condition, $v$
  is not the end of a 2-thread. Moreover, by Lemma
  \ref{lemmaStructure8/3}.3, $v$ is the end of at
  most two 1-threads. 

If $v$ is the end of exactly
  two 1-threads, then it is adjacent to a $^\ge 4$-vertex, by Lemma
   \ref{lemmaStructure8/3}.3. So, $\nch(v) = 3 - 2\cdot \frac{1}{3} +
  \frac{1}{3}  = \frac{8}{3}$ by (R2) and (R3). 

If $v$ is the end 
  of at most one 1-thread, then $\nch(v)\geq 3 - \frac{1}{3} = \frac{8}{3}$.

\item [\emph{Case $d(v)=4$}] By (R2) and (R3), $\nch(v) \geq 4 - 4
  \cdot \frac{1}{3} = \frac{8}{3}$.

\item [\emph{Case $d(v) = 5$}] By Lemma \ref{lemmaStructure8/3}.4, $v$
  is the end of at most four 2-threads. 

If $v$ is the end of exactly four 2-threads ($v$ is
  light), then it is neither adjacent to a $^\leq 3$-vertex, by Lemmas
  \ref{lemmaStructure8/3}.4 and \ref{lemmaStructure8/3}.5, nor to a
  $5_{0,2,2,2,2}$-vertex, by Lemma \ref{lemmaStructure8/3}.6. Hence, $v$ 
  receives $\frac{1}{3}$ by (R3). So, $\nch(v)=5-4\cdot \frac{2}{3} +
  \frac{1}{3} = \frac{8}{3}$. 

If $v$ is the end of exactly three 2-threads,
  then $v$ is adjacent to at most one light vertex, by Lemmas
  \ref{lemmaStructure8/3}.7 and \ref{lemmaStructure8/3}.8. So,
  $\nch(v) \ge 5 - 3 \cdot \frac{2}{3} - \frac{1}{3} \geq
  \frac{8}{3}$. 

Finally, if $v$ is the end of at most two 2-threads, then $\nch(v) \geq 5 - 2 \cdot
  \frac{2}{3} - 3 \cdot \frac{1}{3} \geq \frac{8}{3}$.
\item [\emph{Case $d(v) \ge  6$}] By Lemma \ref{lemmaStructure8/3}.4,
  $v$ is the end of at most $d(v)-1$ 2-threads. 

If $v$ is the end of exactly $d(v)-1$ 2-threads,
  then it is neither adjacent to a $^\leq 3$-vertex, by Lemmas
  \ref{lemmaStructure8/3}.4 and \ref{lemmaStructure8/3}.5, nor to a
  $5_{0,2,2,2,2}$-vertex, by Lemma \ref{lemmaStructure8/3}.6; hence,
  $\nch(v) = d(v) - (d(v)-1)\cdot \frac{2}{3} \geq \frac{8}{3}$. 

If $v$ is the end of exactly $d(v)-2$ 2-threads,
 then  $\nch(v) \ge d(v) - (d(v)-2)\cdot \frac{2}{3} - 2 \cdot \frac{1}{3} \geq
  \frac{8}{3}$. 

Finally, if $v$ is the end of at most $d(v)-3$ 2-threads, then $\nch(v) \geq d(v) - (d(v)-3)\cdot
  \frac{2}{3} - 3 \cdot \frac{1}{3} \geq 3$.

\end{enumerate}
This implies that $\Mad(G)$ $\geq \frac{8}{3}$, which provides the needed contradiction.
\end{proof}

\begin{lemma}\label{structure8/3}
Let $G$ be a graph with  $\Mad(G) < \frac{8}{3}$ and $g(G)\ge 6$. Then there exists a vertex partition $V(G) = F \dot{\cup}
I_1 \dot{\cup} I_2 \dot{\cup} I_3$ such that:
\begin{enumerate}
\item[P1.] $F$ induces a forest. 
\item[P2.] $I_i$ is an independent set such that for all $x,y \in I_i$,
  $d_G(x,y)> 2$ for $i=1,2,3$.
\end{enumerate}
\end{lemma}  

\begin{proof}
Let $G$ be a counterexample with the fewest vertices. By Lemma
\ref{lemmaStructure8/3}, $G$ contains one of the 8 configurations
in Figure \ref{fig8/3}. 
In each configuration,
let $N$ denote the set of black vertices. 
In each case, we 
delete a subgraph $H$ and partition $V(G-H)$ into sets $F'$, $I_1'$, $I_2'$, and $I_3'$ that satisfy $P_1$ and $P_2$.
We now show how to add the deleted vertices to $F'$, $I_1'$, $I_2'$, and $I_3'$ so that the resulting sets still satisfy $P_1$ and $P_2$.

We consider each of the 8 configurations in Figure~\ref{fig8/3}:

\begin{enumerate}
\item [(1)] 
Let $G'=G-v$. By the minimality of $G$, 
  there exists a vertex partition $V(G')=F'\dot{\cup}I_1' \dot{\cup}I_2' \dot{\cup}I_3'$ satisfying P1 and P2. Add $v$ to $F'$.
\item [(2)] Let $G'=G\setminus\{v,w\}$ and 
$V(G')=F'\dot{\cup}I_1' \dot{\cup}I_2' \dot{\cup}I_3'$. 
If at least one of the vertices $u$ and $x$ belongs to an independent set, then 
  add $v$ and $w$ to $F'$.
  So assume that $x,u \in F'$. 
  If all three $y_i$s are in distinct independent sets, then add $v$ and $w$ to $F'$.
  Otherwise, add $v$ to the forest and add $w$ to an independent set that does not contain any $y_i$.

\item [(3)] Let $G'=G\setminus\{v,w,z\}$. If three or more of the vertices
  $u$, $r$, $y_1$, $y_2$ belong to independent sets, then add $v$, $w$, $z$
  to the forest. So suppose that at most two of these vertices belong to
  independent sets. 
Add $w$ to an independent set that 
  does not contain any of $u$, $r$, $y_1$, $y_2$; add all the other vertices to the forest. 
It may happen that $w$ and $x$ are in the same independent set;
  in this case, we can either add $v$, $w$, $z$ to the forest or we can move $x$
  to the forest.
\item [(4)] Let  $G'=G \setminus N$. Add
  $x$ to an independent set that contains neither $v_1$ nor $v_2$ and add
  all the other vertices to the forest.

\item [(5)] Let  $G'=G \setminus \{ \mbox{black vertices except $v$} \}$. Add
  $x$ to an independent set that contains neither $v_1$ nor $v_2$ and add
  all the other vertices to the forest (move $v$ to the forest if necessary).
\item [(6)] Let  $G'=G \setminus N$. Add $v$ to an independent set that 
  contains neither $v_2$ nor $v_3$, add $u$ to an independent set that 
  contains neither $v$ nor $v_1$, and add the all the other vertices to the forest.

\item [(7)] Let $G'=G \setminus N$. 
  Add $r$ to an independent set that contains neither $u_1$ nor $u_2$, and
  add $s$ to an independent set that contains neither $w_1$, nor $u_2$, nor $r$;
if there is no choice for $s$, then add $s$ to the forest.  Add $x$ to an independent set that contains neither $r$ nor $s$, and add all the other vertices to the forest.
\item [(8)] Let  $G'=G \setminus N$. 
  Add $v$ to an independent set that contains neither $v_1$ nor $v_2$,
  add $u$ to an independent set that contains neither $v$ nor $v_3$, and
  add $x$ to an independent set that contains neither $u$ nor $v$.
 Finally, add all the other vertices to the forest.
\end{enumerate}

\end{proof}

\par \noindent \textbf{Proof of Theorem \ref{res} $(iii)$} \\ Let $G$ be a
graph with $\Mad(G)<\frac{8}{3}$ and $g(G)\ge 6$. By Lemma
\ref{structure8/3}, there exists a vertex partition $V(G)=F\dot{\cup} I_1\dot{\cup} I_2 \dot{\cup} I_3$. Now
we color:
\begin{itemize} 
\item $F$ with the colors 1, 2, and 3 (as in Observation~\ref{3-color forest}),
\item $I_1$ with the color 4, 
\item $I_2$ with the color 5, and
\item $I_3$ with the color 6.

\end{itemize} 
This produces a star 6-coloring of $G$.  \hfill{$\Box$}

\begin{observation}
The condition $g(G)\ge 6$ can be dropped. Since there are no 3-threads,
we just use this assumption to ensure that white and black vertices
are distinct in the configurations in Figure~\ref{fig8/3}. 
\end{observation}

\section{Conclusion \label{secConc}}
In this paper, we proved that if $G$ is a graph with $\Mad(G) <
\frac{26}{11}$ (resp. $\Mad(G)< \frac{18}{7}$ and $g(G)\geq 6$, $\Mad(G) < \frac{8}{3}$ and
$g(G)\geq 6$), then $\chi_s(G)\leq 4$ (resp. 5, 6). Are these results
optimal? Let us define a function $f$ as follows.

\begin{definition}
Let $f:\mathbb{N}\rightarrow\mathbb{R}$ be defined by

$$f(n)=\inf\{\Mad(H)\ |\ \chi_s(H)>n\}$$

\end{definition}

Observe that:
$$
f(1)=1.
$$
Actually, if $G$ contains an edge, then $\chi_s(G)>1$; otherwise, $G$
is a independent set, so $\chi_s(G)=1$ and $\Mad(G)=0$.  Note also that:

$$
f(2)=\frac{3}{2}.
$$
If $G$ contains a path of length 3, then $\chi_s(G)> 2$ and
$\Mad(G)\geq \frac{3}{2}$. Otherwise, $G$ is a star forest, so 
$\chi_s(G)\leq 2$ and $\Mad(G)< 2$.  Finally, note that:

$$
f(3)=2.
$$
If $\Mad(G)< 2$, then $G$ is a forest and $\chi_s(G)=3$. In contrast, if $G$ is
a 5-cycle, then $\Mad(G)=2$ and $\chi_s(G)=4$.

\begin{problem}
What is the value of $f(n)$ for $n\geq 4$?
\end{problem}

\begin{figure}[htbp]
\begin{center}
\includegraphics[width=5cm]{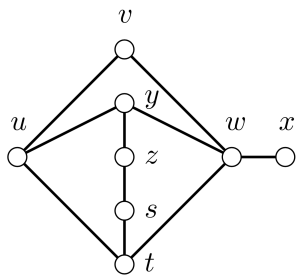}
\caption{A graph $G$ with $\Mad(G)=\frac{18}{7}$ and $\chi_s(G)>4$.} \label{exampleFig}
\end{center}
\end{figure}

From Theorem~\ref{res}.1, we know that:

$$
\frac{26}{11}\leq f(4).
$$


The graph $G$, in Figure \ref{exampleFig}, has $\Mad(G)=\frac{18}{7}$ 
and $\chi_s(G)>4$, so $G$ gives an upper bound on $f(4)$.
Hence:

$$\frac{26}{11}=\frac{182}{77}\le f(4)\le \frac{198}{77}=\frac{18}{7}$$

\section{Acknowledgements}
Thanks to Craig Timmons and Andr\'e K\"undgen for helpful discussions.

\end{document}